\documentclass{amsart}
\usepackage{amssymb}
\usepackage{graphicx}
\usepackage{amscd}

\newtheorem{theorem}{Theorem}
\theoremstyle{plain}

\newtheorem{lemma}{Lemma}

\newtheorem{proposition}{Proposition}

\numberwithin{equation}{section}

\input{tcilatex}

\begin{document}
\title[Topological Structure on AG-groupoids]{Topological Structure on
Abel-Grassmann's Groupoids}
\author{Qaiser Mushtaq and Madad Khan}
\address{Department of mathematics, Quaid-i-Azam University, Islamabad,
Pakistan.}
\email{qmushtaq@isb.apollo.net.pk}
\email{madadmath@yahoo.com}
\subjclass[2000]{$20$M$10$ and $20$N$99$}

\begin{abstract}
In this paper we have discussed the ideals in Abel Grassmann's groupoids and
construct their topologies.
\end{abstract}

\maketitle

\section{Introduction}

An Abel-Grassmann's groupoid $($AG-groupoid$)$ $[3]$ is a groupoid $S$ with
left invertive law

\begin{equation}
(ab)c=(cb)a\text{, for all }a\text{, }b\text{, }c\in S\text{.}  \tag{1}
\end{equation}

Every AG-groupoid $S$ satisfy the medial law $[2]$

\begin{equation}
(ab)(cd)=(ac)(bd)\text{,}\ \text{for all }a,b,c,d\in S\text{.}  \tag{2}
\end{equation}

In every AG-groupoid with left identity the following law $[3]$ holds%
\begin{equation}
(ab)(cd)=(db)(ca)\text{,}\ \text{for all }a,b,c,d\in S\text{.}  \tag{3}
\end{equation}

Many characteristics of several non-associative AG-groupoids similar to a
commutative semigroup.

The aim of this note is to define the topological spaces using ideal theory.
Several ideals concerning the number of occurrence of topological spaces in
AG-groupoids. The topological spaces formation guarantee for the
preservation of finite intersection and arbitrary union between the set of
ideals and the open subsets of resultant topologies.

A subset $I$ of an AG-groupoid $S$ is called a right (left) ideal if $%
IS\subseteq I$ $(SI\subseteq I)$, and is called an ideal if it is two sided
ideal, if $I$ is a left ideal of $S$ then $I^{2}$ becomes an ideal of $S$.
By a bi-ideal of an AG-groupoid $S$, we mean a sub AG-groupoid $B$ of $S$
such that $(BS)B\subseteq B$. It is easy to note that each right ideal is a
bi-ideal. If $S$ has a left identity then it is not hard to show that $B^{2}$
is a bi-ideal of $S$ and $B^{2}\subseteq SB^{2}=B^{2}S$. If $E(B_{S})$
denote the set of all idempotents subsets of $S$ with left identity $e$,
then $E(B_{S})$ form a semilattice structure also if $C=C^{2}$ then $%
(CS)C\in E(B_{S})$. The intersection of any set of bi-ideals of an
AG-groupoid $S$ is either empty or a bi-ideal of $S$. Also the intersection
of prime bi-ideals of an AG-groupoid $S$ is a semiprime bi-ideal of $S$.

If $S$ is an AG-groupoid with left identity $e$, and assume that $%
a^{3}=a^{2}a$ then 
\begin{equation*}
(x_{1}^{m}x_{2}^{n})(x_{3}^{q}x_{4}^{r})=(x_{p(1)}^{m}x_{p(2)}^{n})(x_{P(3)}^{q}x_{P(4)}^{r})%
\text{, for }m,n,q,r\geq 2\text{,}
\end{equation*}%
where $\{p(1),p(2),p(3),p(4)\}$ means any permutation on the set $%
\{1,2,3,4\} $. As a consequence $%
(x_{1}x_{2}x_{3}x_{4})^{k}=(x_{p(1)}x_{p(2)}x_{P(3)}x_{P(4)})^{k}$, for $%
k\geq 2$. The result can be generalized for finite numbers of elements of $S$%
. If $0\in S$, then $0s=s0=0$, for all $s$ in $S$.

\begin{proposition}
Let $T$ be a left ideal and $B$ is a bi-ideal of an AG-groupoid $S$ with
left identity , then $BT$ and $T^{2}B$ are bi-ideals of $S$.
\end{proposition}

\begin{proof}
Using $(2$), we get%
\begin{eqnarray*}
((BT)S)(BT) &=&((BT)B)(ST)\subseteq ((BS)B)T\subseteq BT, \\
\text{also }(BT)(BT) &=&(BB)(TT)\subseteq BT\text{.}
\end{eqnarray*}%
Hence $BT$ is a bi-ideal of $S$. Now using $(2)$, we obtain%
\begin{eqnarray*}
((T^{2}B)S)(T^{2}B) &=&((T^{2}S)(BS))(T^{2}B)\subseteq (T^{2}(BS))(T^{2}B) \\
&=&(T^{2}T^{2})((BS)B)\subseteq T^{2}B\text{, also} \\
(T^{2}B)(T^{2}B) &=&(T^{2}T^{2})(BB)\subseteq T^{2}B\text{.}
\end{eqnarray*}%
Hence $T^{2}B$ is a bi-ideal of $S$.
\end{proof}

\begin{proposition}
The product of two bi-ideals of an AG-groupoid $S$ with left identity is a
bi-ideal of $S$.
\end{proposition}

\begin{proof}
Using (2), we get%
\begin{eqnarray*}
((B_{1}B_{2})S)(B_{1}B_{2})
&=&((B_{1}B_{2})(SS))(B_{1}B_{2})=((B_{1}S)(B_{2}S))(B_{1}B_{2}) \\
&=&((B_{1}S)B_{1})((B_{2}S)B_{2})\subseteq B_{1}B_{2}\text{.}
\end{eqnarray*}
\end{proof}

If $B_{1}$and $B_{2}$ are non-empty, then $B_{1}B_{2}$ and $B_{2}B_{1}$ are
connected bi-ideals. Also the above Proposition leads us to easy
generalizations that is, if $B_{1}$, $B_{2}$, $B_{3}$,$...$ and $B_{n}$ are
bi-ideals of an AG-groupoid $S$ with left identity, then%
\begin{equation*}
(...((B_{1}B_{2})B_{3})...)B_{n}\text{ and }%
(...((B_{1}^{2}B_{2}^{2})B_{3}^{2})...)B_{n}^{2}
\end{equation*}%
are bi-ideals of $S$, consequently the set $\complement (S_{B})$ of
bi-ideals form an AG-groupoid.

If $S$ an AG-groupoid with left identity $e$ then $\langle a\rangle _{L}=Sa$%
, $\langle a\rangle _{R}=aS$ and $\langle a\rangle _{S}=(Sa)S$ are bi-ideals
of $S$. Now it is not hard to show that $\langle ab\rangle _{L}=\langle
a\rangle _{L}\langle b\rangle _{L}$, $\langle ab\rangle _{R}=\langle
a\rangle _{R}\langle b\rangle _{R}$, and $\langle ab\rangle _{R}=\langle
b\rangle _{L}\langle a\rangle _{L}$, from these it can be deduce that $%
\langle a\rangle _{R}\langle b\rangle _{R}=\langle b\rangle _{L}\langle
a\rangle _{L}$ and $\langle a\rangle _{L}\langle b\rangle _{L}=\langle
b\rangle _{R}\langle a\rangle _{R}$. Also $\langle a\rangle _{L}\langle
b\rangle _{R}=\langle b\rangle _{L}\langle a\rangle _{R}$,$\ \langle
a^{2}\rangle _{L}=\langle a\rangle _{L}^{2}$, $\langle a^{2}\rangle
_{R}=\langle a\rangle _{R}^{2}$, $\langle a^{2}\rangle _{L}=\langle
a^{2}\rangle _{R}$ and $\langle a\rangle _{L}=\langle a\rangle _{R}$ (if $a$
is an idempotent), consequently $\langle a^{2}\rangle _{L}=\langle
a^{2}\rangle _{R}$. It is easy to show that $\langle a\rangle
_{R}a^{2}=a^{2}\langle a\rangle _{L}$.

\begin{lemma}
If $B$ is an idempotent bi-ideal of an AG-groupoid $S$ with left identity,
then $B$ is an ideal of $S$.
\end{lemma}

\begin{proof}
Using (1),%
\begin{equation*}
BS=(BB)S=(SB)B=(SB^{2})B=(B^{2}S)B=(BS)B\text{,}
\end{equation*}%
and every right ideal in $S$ with left identity is left.
\end{proof}

\begin{lemma}
If $B$ is a proper bi-ideal of an AG-groupoid $S$ with left identity $e$,
then $e\notin B$.
\end{lemma}

\begin{proof}
Let $e\in B$, since $sb=(es)b\in B$, now using (1), we get $s=(ee)s=(se)e\in
(SB)B\subseteq B$.
\end{proof}

It is easy to note that $\{x\in S:(xa)x=e\}\nsubseteq B.$

\begin{proposition}
If $A$, $B$ are bi-ideals of an AG-groupoid $S$ with left identity, then the
following assertions are equivalent.

$(i)$ Every bi-ideal of $S$ is idempotent,

$(ii)$ $A\cap B=AB$, and

$(iii)$ the ideals of $S$ form a semilattice $(L_{S},\wedge )$ where $%
A\wedge B=AB$.
\end{proposition}

\begin{proof}
$(i)\Rightarrow (ii)$: Using Lemma $1$, it is easy to note that $AB\subseteq
A\cap B$. Since $A\cap B\subseteq A,$ $B$ implies $(A\cap B)^{2}\subseteq AB$%
, hence $A\cap B\subseteq AB$.

$(ii)\Rightarrow (iii)$: $A\wedge B=AB=A\cap B=B\cap A=B\wedge A$ and $%
A\wedge A=AA=A\cap A=A$. Similarly, associativity follows. Hence $%
(L_{S},\wedge )$ is a semilattice.

$(iii)\Rightarrow (i)$: 
\begin{equation*}
A=A\wedge A=AA\text{.}
\end{equation*}
\end{proof}

A bi-ideal $B$ of an AG-groupoid $S$ is called a prime bi-ideal if $%
B_{1}B_{2}\subseteq B$ implies either $B_{1}\subseteq B$ or $B_{2}\subseteq
B $ for every bi-ideal $B_{1}$ and $B_{2}$ of $S$. The set of bi-ideals of $%
S $ is totally ordered under inclusion if for all bi-ideals $I$, $J$ either $%
I\subseteq J$ or $J\subseteq I$.

\begin{theorem}
Each bi-ideal of an AG-groupoid $S$ with left identity is prime if and only
if it is idempotent and the set of bi-ideals of $S$ is totally ordered under
inclusion.
\end{theorem}

\begin{proof}
Assume that each bi-ideal of $S$ is prime. Since $B^{2}$ is an ideal and so
is prime which implies that $B\subseteq B^{2}$, hence $B$ is idempotent.
Since $B_{1}\cap B_{2}$ is a bi-ideal of $S$ (where $B_{1}$ and $B_{2}$ are
bi-ideals of $S$) and so is prime, now by Lemma $1$, either $B_{1}\subseteq
B_{1}\cap B_{2}$ or $B_{2}\subseteq B_{1}\cap B_{2}$ which further implies
that either $B_{1}\subseteq B_{2}$ or $B_{2}\subseteq B_{1}$. Hence the set
of bi-ideals of $S$ is totally ordered under inclusion.

Conversely, assume that every bi-ideal of $S$ is idempotent and the set of
bi-ideals of $S$ is totally ordered under inclusion. Let $B_{1}$, $B_{2}$
and $B$ be the bi-ideals of $S$ with $B_{1}B_{2}\subseteq B$ and without
loss of generality assume that $B_{1}\subseteq B_{2}$. Since $B_{1}$ is an
idempotent, so $B_{1}=B_{1}B_{1}\subseteq B_{1}B_{2}\subseteq B$ implies
that $B_{1}\subseteq B$ and hence each bi-ideal of $S$ is prime.
\end{proof}

A bi-ideal $B$ of an AG-groupoid $S$ is called strongly irreducible bi-ideal
if $B_{1}\cap B_{2}\subseteq B$ implies either $B_{1}\subseteq B$ or $%
B_{2}\subseteq B$ for every bi-ideal $B_{1}$ and $B_{2}$ of $S$.

\begin{theorem}
Let \DH\ be the set of all bi-ideals of an AG-groupoid $S$ with zero and $%
\Omega $ be the set of all strongly irreducible proper bi-ideals of $S$,
then $\Gamma (\Omega )=\{O_{B}:B\in \text{\DH }\}$, form a topology on the
set $\Omega $, where $O_{B}=\{J\in \Omega ;B\nsubseteq J\}$ and $\phi :$
bi-ideal$(S)\longrightarrow \Gamma (\Omega )$ preserves finite intersection
and arbitrary union between the set of bi-ideals of $S$ and open subsets of $%
\Omega $.
\end{theorem}

\begin{proof}
As $\{0\}$ is a bi-ideal of $S$, and $0$ belongs to every bi-ideal of $S$,
then $O_{\{0\}}=\{J\in \Omega ,\{0\}\nsubseteq J\}=\{\}$, also $O_{S}=\{J\in
\Omega ,$ $S\nsubseteq J\}=\Omega $ which is the first axiom for the
topology. Let $\{O_{B_{\alpha }}:\alpha \in I\}\subseteq \Gamma (\Omega )$,
then $\cup O_{B_{\alpha }}=\{J\in \Omega ,$ $B_{\alpha }\nsubseteq J$, for
some $\alpha \in I\}=\{J\in \Omega ,$ $<\cup B_{\alpha }>$ $\nsubseteq
J\}=O_{\cup B_{\alpha }}$, where\ $<\cup B_{\alpha }>$\ is a bi-ideal of $S$
generated by $\cup B_{\alpha }$. Let $O_{B_{1}}$ and $O_{B_{2}}\in \Gamma
(\Omega )$, if $J\in O_{B_{1}}\cap O_{B_{2}}$, then $J\in \Omega $ and $%
B_{1}\nsubseteq J$, $B_{2}\nsubseteq J$. Suppose $B_{1}\cap B_{2}\subseteq J$%
, this implies that either $B_{1}\subseteq J$ or $B_{2}\subseteq J$, which
leads us to a contradiction. Hence $B_{1}\cap B_{2}\nsubseteq J$ which
further implies that $J\in O_{B_{1}\cap B_{2}}$. Thus $O_{B_{1}}\cap
O_{B_{2}}\subseteq O_{B_{1}\cap B_{2}}$. Now if $J\in O_{B_{1}\cap B_{2}}$,
then $J\in \Omega $ and $B_{1}\cap B_{2}\nsubseteq J$. Thus $J\in O_{B_{1}}$
and $J\in O_{B_{2}}$, therefore $J\in O_{B_{1}}\cap O_{B_{2}}$, which
implies that $O_{B_{1}\cap B_{2}}\subseteq O_{B_{1}}\cap O_{B_{2}}$. Hence $%
\Gamma (\Omega )$ is the topology on $\Omega $. Define $\phi :$ bi-ideal$%
(S)\longrightarrow \Gamma (\Omega )$ by $\phi (B)=O_{B}$, then it is easy to
note that $\phi $ preserves finite intersection and arbitrary union.
\end{proof}

An ideal $P$ of an AG-groupoid $S$ is called prime if $AB\subseteq P$
implies that either $A\subseteq P$ or $B\subseteq P$ for all ideals $A$ and $%
B$ in $S$.

Let $P_{S}$ denote the set of proper prime ideals of an AG-groupoid $S$
absorbing $0$. For an ideal $I$ of $S$ define the set $\Theta _{I}=\{$ $J\in
P_{S}:I\nsubseteq J\}$ and $\Gamma (P_{S})=\{\Theta _{I},$ $I$ is an ideal
of $S\}$.

\begin{theorem}
Let $S$ is an AG-groupoid with $0$. The set $\Gamma (P_{S})$\ constitute a
topology on the set $P_{S}$.
\end{theorem}

\begin{proof}
Let $\Theta _{I_{1}}$, $\Theta _{I_{2}}\in \Gamma (P_{S})$, if $J\in \Theta
_{I_{1}}\cap \Theta _{I_{2}}$, then $J\in P_{S}$ and $I_{1}\nsubseteq J$ and 
$I_{2}\nsubseteq J$. Let $I_{1}\cap I_{2}\subseteq J$ which implies that
either $I_{1}\subseteq J$ or $I_{2}\subseteq J$, which is contradiction.
Hence $J\in \Theta _{I_{1}\cap I_{2}}$. Similarly $\Theta _{I_{1}\cap
I_{2}}\subseteq \Theta _{I_{1}}\cap \Theta _{I_{2}}$. The remaining proof
follows from Theorem 2.
\end{proof}

The assignment $I\longrightarrow \Theta _{I}$ preserves finite intersection
and arbitrary union between the ideal$(S)$ and their corresponding open
subsets of $\Theta _{I}$.

Let $P$ be a left ideal of an AG-groupoid $S$. $P$ is called quasi-prime if
for left ideals $A$, $B$ of $S$ such that $AB\subseteq P$, we have $%
A\subseteq P$ or $B\subseteq P$.

\begin{theorem}
If $S$ is an AG-groupoid $S$ with left identity $e$, then a left ideal $P$
of $S$ is quasi-prime if and only if $(Sa)b\subseteq P$ implies that either $%
a\in P$ or $b\in P$.
\end{theorem}

\begin{proof}
Let $P$ be a left ideal of an AG-groupoid $S$ with left identity $e$. Now
assume that $(Sa)b\subseteq P$, then%
\begin{eqnarray*}
S((Sa)b) &\subseteq &SP\subseteq P\text{, that is} \\
S((Sa)b) &=&(Sa)(Sb)
\end{eqnarray*}

Hence, either $a\in P$ or $b\in P$.

Conversely, assume that $AB\subseteq P$ where $A$ and $B$ are left of $S$
such that $A\nsubseteq P$. Then there exists $x\in A$ such that $x\notin P$.
Now using the hypothesis we get $(Sx)y\subseteq (SA)B\subseteq AB\subseteq P$
for all $y\in B$. Since $x\notin P$, so by hypothesis, $y\in P$ for all $%
y\in B$, we obtain $B\subseteq P$. This shows that $P$ is quasi-prime.
\end{proof}

An AG-groupoid $S$ is called an anti-rectangular if $a=(ba)b$, for all $a$,$%
b $ in $S$. It is easy to see that $S=S^{2}$. In the following results for
an anti-rectangular AG-groupoid $S$, $e\notin S$.

\begin{proposition}
If $A$ and $B$ are the ideals of an anti-rectangular AG-groupoid $S$, then $%
AB$ is an ideal.
\end{proposition}

\begin{proof}
Using (2), we get%
\begin{eqnarray*}
(AB)S &=&(AB)(SS)=(AS)(BS)\subseteq AB,\text{ also} \\
S(AB) &=&(SS)(AB)=(SA)(SB)\subseteq AB
\end{eqnarray*}

which shows that $AB$ is an ideal.
\end{proof}

Consequently, if $I_{1}$, $I_{2}$, $I_{3}$,$...$ and $I_{n}$ are ideals of $%
S $, then 
\begin{equation*}
(...((I_{1}I_{2})I_{3})...I_{n})\text{ and }%
(...((I_{1}^{2}I_{2}^{2})I_{3}^{2})...I_{n}^{2})
\end{equation*}%
are ideals of $S$ and the set $\circledS _{I}$ of ideals of $S$ form an
anti-rectangular AG-groupoid.

\begin{lemma}
Any subset of an anti-rectangular AG-groupoid $S$ is left ideal if and only
if it is right.
\end{lemma}

\begin{proof}
Let $I$ be a right ideal of $S$, then using (1), we get, $%
si=((xs)x)i=(ix)(xs)\in I.$

Conversely, suppose that $I$ be a left ideal of $S$, then using (1), we get, 
$is=((yi)y)s=(sy)(yi)\in I.$
\end{proof}

It is fact that $SI=IS$. From above Lemma we remark that, each quasi prime
ideals becomes prime in an anti-rectangular AG-groupoid.

\begin{lemma}
If $I$ is an ideal of an anti-rectangular AG-groupoid $S$ then, $H(a)=\{x\in
S:(xa)x=a$, for $a\in I\}\subseteq I$.
\end{lemma}

\begin{proof}
Let $y\in H(a)$, then $y=(ya)y\in (SI)S\subseteq I$. Hence $H(a)\subseteq I$.

Also $H(a)=\{x\in S:(xa)x=x$, for $a\in I\}\subseteq I$.
\end{proof}

An ideal $I$ of an AG-groupoid $S$ is called an idempotent if $I^{2}=I$. An
AG-groupoid $S$ is said to be fully idempotent if every ideal of $S$ is
idempotent.

\begin{proposition}
If $S$ is an anti-rectangular AG-groupoid and $A$, $B$ are ideals of $S$,
then the following assertions are equivalent.

$(i)$ $S$ is fully idempotent,

$(ii)$ $A\cap B=AB$, and

$(iii)$ the ideals of $S$ form a semilattice $(L_{S},\wedge )$ where $%
A\wedge B=AB$.
\end{proposition}

The proof follows from Proposition $3$.

The set of ideals of $S$ is totally ordered under inclusion if for all
ideals $I$, $J$ either $I\subseteq J$ or $J\subseteq I$ and denoted by ideal$%
(S)$.

\begin{theorem}
Every ideal of an anti-rectangular AG-groupoid $S$ is prime\ if and only it
is idempotent and ideal$(S)$ is totally ordered under inclusion.
\end{theorem}

The proof follows from Theorem $1$.

\end{document}